\documentclass[12pt]{article}
\title{Minimal orbits of promotion}
\author{Kevin Purbhoo\thanks{
Research of Purbhoo was 
supported by an NSERC Discovery grant.} %
\ and 
Donguk Rhee
}

\usepackage{latexsym, amsthm, amsmath, enumerate}
\usepackage{graphicx, xspace, ifthen, rotating, pict2e}
\usepackage[svgnames]{xcolor}
\usepackage[charter]{mathdesign}
\usepackage{algpseudocode, algorithmicx}

\usepackage{young}
\ysetshade{Blue!30}
\ysetaltshade{Green!50}
\YSetShade{Green!30}
\YSetAltShade{Purple!45}

\newlength\circlesize
\allowdisplaybreaks[1]

\newcommand{\bigmid}{\ \big|\ }

\newcommand{\ZZ}{\mathbb{Z}}

\newcommand{\Hw}{{\widehat w}}
\newcommand{\HU}{{\widehat U}}
\newcommand{\Hsigma}{{\widehat \sigma}}
\newcommand{\la}[1]{{\lambda_{#1}}}
\newcommand{\veela}[1]{{\smash{\lambda^\vee}\!\!\!_{#1}}}
\newcommand{\primela}[1]{{\smash{\lambda'}\!_{#1}}}

\newcommand{\nth}{\ensuremath{^\text{th}}\xspace}

\newcommand{\SL}{\mathrm{SL}}

\newcommand{\SYT}{\mathsf{SYT}}

\newcommand{\calO}{\mathcal{O}}

\newcommand{\aug}{\mathrm{aug}}

\newcommand{\promote}{\partial}
\newcommand{\trunc}{\mathrm{Trunc}}

\newcommand{\identity}{\mathrm{id}}

\newcommand{\one}{{\mathchoice%
{\yng[bb][1ex](1)}
{\yng[bb][1ex](1)}
{\yng[bb][.7ex](1)}
{\yng[bb][.7ex](1)}
}}
\newcommand{\diag}{{\yeveryframe{\yshade}\one}}
\newcommand{\Rect}{{\setlength{\yframethickness}{.7pt}\yhwratio{2:3}%
\mathchoice%
{{\yng[bb][1.1ex](1)}}
{{\yng[bb][1.1ex](1)}}
{{\yng[bb][.7ex](1)}}
{{\yng[bb][.7ex](1)}}
}}
\newcommand{\Hone}{{\widehat{\one}}}
\newcommand{\boxseq}{{\setlength{\yframethickness}{.25ex}
{\yng[bb][.75ex](1)}\,}}
\newcommand{\SYTRect}{\SYT\big(\Rect\big)}

\newenvironment{packedenumi}{
\begin{enumerate}[(i)]
  \setlength{\itemsep}{0pt}
}{\end{enumerate}}

\newenvironment{packedenum}{
\begin{enumerate}
  \setlength{\itemsep}{0pt}
}{\end{enumerate}}

\newtheorem{lemma}{Lemma}
\newtheorem{theorem}[lemma]{Theorem}

\newtheorem{proposition}[lemma]{Proposition}

\theoremstyle{definition}

\newtheorem{algorithm}{Algorithm}

\definecolor{DarkBlue}{rgb}{0, 0.1, 0.55}
\definecolor{DarkRed}{rgb}{0.45, 0, 0}
\newcommand{\defn}[1]{\textbf{#1}}

\begin{document}
\maketitle

\begin{abstract}
We give a
bijection between the symmetric group $S_n$, and the set of
standard Young tableaux of rectangular shape $m^n$, $m \geq n$,
that have order $n$ under \emph{jeu de taquin} promotion.
\end{abstract}


\section{Introduction}

Fix positive integers $m \geq n$, and let $\Rect$ be either
the $m \times n$ rectangle, or the $n \times m$ rectangle.
The \defn{promotion} map $\promote: \SYTRect \to \SYTRect$
defines an action of $(\ZZ, +)$ 
on the set of standard Young tableaux of shape $\Rect$\,.  
For $T \in \SYT(\Rect)$,
$\promote T$ is computed by deleting the entry $1$ from $T$, 
decrementing each entry by $1$, rectifying, and finally adding
an entry $mn$ in the lower-right corner.
The interest in this action stems from its
connections to geometry and representation theory, and its
striking combinatorial properties,
(see \cite{FK, Pur-ribbon, Rho, Sta, Wes}).

Let $\calO_r := \{T \in \SYTRect \mid \promote^r T = T\}$ denote
the set of tableaux whose order under promotion divides $r$.
By a theorem of Haiman~\cite{Hai},
$\promote^{mn} T = T$ for all $T \in \SYTRect$; hence
$\calO_r$ is empty if $r$ is coprime to $mn$.
It is also not hard to see that $\calO_r$ is empty for $r < n$.
The minimal orbits of promotion, therefore, have order $n$.

The action of promotion on $\SYTRect$ exhibits a 
cyclic sieving phenomenon, as defined in \cite{SRW}: 
we have $|\calO_r| = F(\zeta^r)$, 
where
$F(q)$ is a $q$-analogue of the hook 
length formula for $|\SYTRect|$, 
and $\zeta$ is a primitive $(mn)$\nth root of unity.
The quantity $F(\zeta^r)$ appears in
a number of other places in representation theory and combinatorics,
which proffers a variety of avenues of proof for this cyclic sieving
theoerm.  It was first proved by Rhoades 
using Kazhdan-Lusztig 
theory \cite{Rho}. Subsequently other proofs were found using 
representation theory 
of $\SL_n$ 
\cite{Wes}, 
and the geometry of the Grassmannian \cite{Pur-ribbon} and
the affine Grassmannian~\cite{FK}.
Simpler, more combinatorial proofs are known in special cases 
when $n =2, 3$ \cite{PPR}.
The survey~\cite{Sag-CSP} discusses of a number of related results.
However, at present there is no known combinatorial proof in general, 
nor any proof that gives an \emph{effective}
description of the sets $\calO_r$.

The purpose of this paper is to give an explicit combinatorial 
construction 
of the orbits in $\calO_n$, i.e. the minimal orbits of promotion.  
Using Rhoades' cyclic sieving theorem, 
one can compute that $|\calO_n| = n!$.
Our main result gives a bijection between the symmetric 
group $S_n$ and $\calO_n$.  Under this bijection promotion corresponds
to right-multiplication by the $n$-cycle $(1\ n\ n{-}1\ \dots\ 2)$.
There are a number of arbitrary choices involved in constructing the 
bijection, and much of the proof is concerned with showing that the 
construction is in fact well-defined.

Choose a skew shape $\la+/\la- \subset \Rect$\,, consisting
of $n$ boxes $\diag_1, \diag_2, \dots, \diag_n$,  such that 
$\diag_{i+1}$ is strictly above and strictly right 
of $\diag_i$, for $i = 1, \dots, n-1$.
We call $\la+/\la-$ a \defn{diagonal} of $\Rect$\,,
(see Figure~\ref{fig:diagonal}).
For each permutation $w \in S_n$, we define a tableau $T_w^\la+$ 
of shape $\la+$, using a procedure similar to rectification.  
In the following algorithm, $T$ is a tableau
under construction.  If $\one$ is a box of $\la+$, we
write $\one \in \la+$, and $T[\one]$ 
denotes the entry of $T$ in box $\one$.

\begin{figure}[tb]
\begin{center}
\begin{young} 
 & & & &!& \\
 & & &!& & \\
 & &!& & & \\
!& & & & & \\
\end{young}
\end{center}
\caption{An example of a diagonal of $\protect\Rect$\,: here 
$n = 4$, $m=6$, $\la+ = 5431$ and $\la- = 432$.}
\label{fig:diagonal}
\end{figure}

\begin{algorithm}
\label{alg:forward} \emph{INPUT:} A permutation $w \in S_n$.
\begin{algorithmic}
\State
Begin with
$T[\diag_i] := w(i)$, for $i = 1, \dots, n$, leaving
all boxes of $\la-$ unfilled; 
\While{$\mathrm{shape}(T) \neq \la+$}
\State  Let $\mu \subset \la+$ be the unfilled boxes to the left of $T$; 
\State  Choose any corner box $\one \in \mu$; 
\State  Let $T'$ be the tableau obtained by sliding $\one$ through $T$; 
\State  If the final position of the sliding path is $\diag_i$, then set 
$T'[\diag_i] := T[\diag_i] + n$; 
\State  Set $T := T'$;
\EndWhile
\State \Return the resulting tableau, $T_w^\la+ := T$.
\end{algorithmic}
\end{algorithm}

Algorithm~\ref{alg:forward} is an extension of the procedure
for computing the insertion tableau of $w$ via rectification. 
As we slide empty boxes of $\la-$ through $T$, we also refill
the boxes of $\la+/\la-$.
Since the new entries are all greater
than $n$, the subtableau of $T_w^\la+$ formed
by entries $1, \dots, n$ will be the insertion tableau of $w$.
An example of the algorithm is given in Figure~\ref{fig:forward}.
Note that $T_w^\la+$ is not a standard Young tableau, since the
entries are not $\{1, \dots, |\la+|\}$.

\begin{figure}[tb]
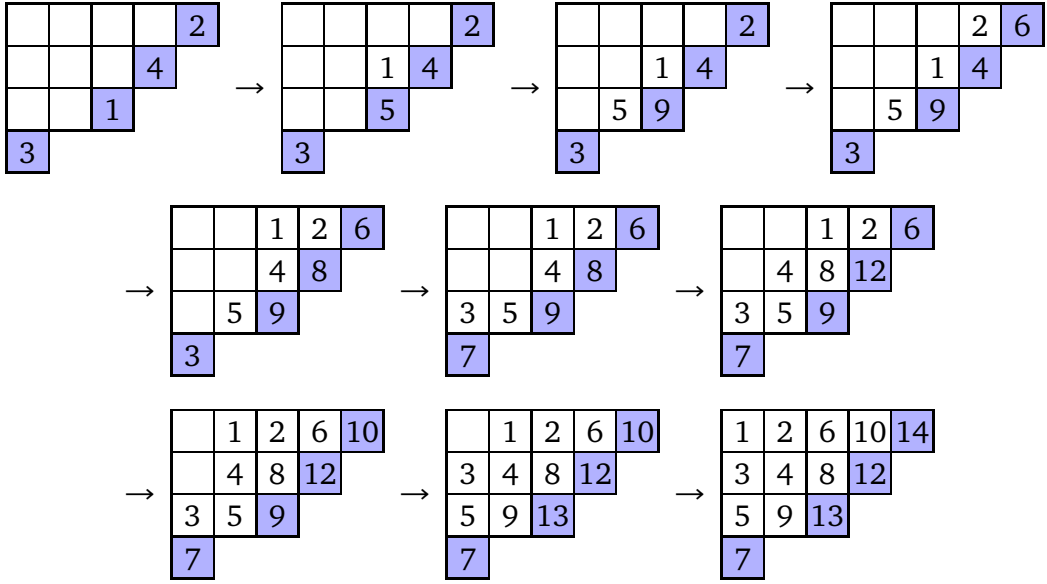

\begin{gather*}
{\begin{young}[c]
   &   &   &   &!2 \\
   &   &   &!4  \\
   &   &!1 \\
!3
\end{young}}
\ \to \ %
{\begin{young}[c]
   &   &   &   &!2 \\
   &   & 1 &!4  \\
   &   &!5 \\
!3
\end{young}}
\ \to \ %
{\begin{young}[c]
   &   &   &   &!2 \\
   &   & 1 &!4  \\
   & 5 &!9 \\
!3
\end{young}}
\ \to \ %
{\begin{young}[c]
   &   &   & 2 &!6 \\
   &   & 1 &!4  \\
   & 5 &!9 \\
!3
\end{young}} \\[1.5ex]
\ \to \ %
{\begin{young}[c]
   &   & 1 & 2 &!6 \\
   &   & 4 &!8  \\
   & 5 &!9 \\
!3
\end{young}}
\ \to \ %
{\begin{young}[c]
   &   & 1 & 2 &!6 \\
   &   & 4 &!8  \\
 3 & 5 &!9 \\
!7
\end{young}}
\ \to \ %
{\begin{young}[c]
   &   & 1 & 2 &!6 \\
   & 4 & 8 &!12  \\
 3 & 5 &!9 \\
!7
\end{young}} \\[1.5ex]
\ \to \ %
{\begin{young}[c]
   & 1 & 2 & 6 &!10 \\
   & 4 & 8 &!12  \\
 3 & 5 &!9 \\
!7
\end{young}}
\ \to \ %
{\begin{young}[c]
   & 1 & 2 & 6 &!10 \\
 3 & 4 & 8 &!12  \\
 5 & 9 &!13 \\
!7
\end{young}}
\ \to \ %
{\begin{young}[c]
 1 & 2 & 6 & 10 &!14 \\
 3 & 4 & 8 &!12  \\
 5 & 9 &!13 \\
!7
\end{young}}
\end{gather*}
\caption{Construction of $T_w^\la+$, with $w = 3142$, and 
$\la+ = 5431$.}
\label{fig:forward}
\end{figure}

\begin{theorem}
\label{thm:welldefined}
The definition of $T_w^\la+$ is independent of the choices in 
Algorithm~\ref{alg:forward}.
\end{theorem}

Theorem~\ref{thm:welldefined} is analogous to the well known fact 
that ordinary rectification is well-defined \cite{Sch}.
However, despite the similarity between Algorithm~\ref{alg:forward}
and rectification, one cannot easily deduce one from the other.
We discuss some of the difficulties in Section~\ref{sec:remarks}.

Similarly, we define a tableau $T_w^{\Rect/\la-}$ of
shape $\Rect/\la-$, using reverse slides.

\begin{algorithm}
\label{alg:reverse} \emph{INPUT:} A permutation $w \in S_n$.
\begin{algorithmic}
\State
Begin with $T[\diag_i] := w(i) + (m-1)n$, for $i = 1, \dots, n$, 
and all boxes of $\Rect/\la+$ unfilled; 
\While{$\mathrm{shape}(T) \neq \Rect/\la-$}
\State  Let $\mu \subset \Rect/\la-$ be the unfilled boxes to the 
  right of $T$; 
\State  Choose any corner box $\one$ of $\mu$; 
\State  Let $T'$ be the tableau obtained by reverse-sliding $\one$ through $T$; 
\State  If final position of the sliding path is $\diag_i$, then set 
$T'[\diag_i] := T[\diag_i] - n$; 
\State  Set $T := T'$; 
\EndWhile
\State \Return the resulting tableau, $T_w^{\Rect/\la-} := T$.
\end{algorithmic}
\end{algorithm}

\begin{figure}[tb]
\begin{gather*}
{\begin{young}[c]
 ,  & , & ,  & ,  &!22 & \\
 ,  & , & ,  &!24 &    & \\
 ,  & , &!21 &    &    & \\
!23 &   &    &    &    &
\end{young}}
\ \to \ %
{\begin{young}[c]
 ,  & , & ,  & ,  &!22 & \\
 ,  & , & ,  &!20 &    & \\
 ,  & , &!21 & 24 &    & \\
!23 &   &    &    &    &
\end{young}}
\ \to \ %
{\begin{young}[c]
 ,  & , & ,  & ,  &!22 & \\
 ,  & , & ,  &!20 &    & \\
 ,  & , &!21 & 24 &    & \\
!19 & 23 &    &    &    &
\end{young}} \\[1.5ex]
\ \to \ %
{\begin{young}[c]
 ,  & , & ,  & ,  &!22 & \\
 ,  & , & ,  &!20 &    & \\
 ,  & , &!21 & 24 &    & \\
!15 & 19 & 23 &    &    &
\end{young}}
\ \to \ \  \dots \ \ \to \ %
{\begin{young}[c]
 ,  & , & ,  & ,  & !14  &18  \\
 ,  & , & ,  &!12 & 16 & 20  \\
 ,  & , &!13 & 17 & 21 & 22  \\
!7  &11 & 15 & 19 & 23 & 24
\end{young}}
\end{gather*}
\caption{Construction of $T_w^{\protect\Rect/\la-}$, with 
$w = 3142$, $m = 6$, and $\la+ = 5431$.}
\label{fig:reverse}
\end{figure}

An example is given in Figure~\ref{fig:reverse}.
Since Algorithm~\ref{alg:reverse} is essentially 
``Algorithm~\ref{alg:forward} turned 
upside-down'', Theorem~\ref{thm:welldefined} implies that 
the definition of $T_w^{\Rect/\la-}$ is independent of choices.
We combine these two constructions to
produce a tableau $T_w$ of shape $\Rect$\,:
for each box $\one \in \Rect$\,, let
\[
T_w[\one] := 
\begin{cases}
T_w^\la+[\one] &\quad \text{if }\one \in \la+ \\
T_w^{\Rect/\la-}[\one] &\quad \text{otherwise.}
\end{cases}
\]
For example, for $w=3142$, $m=6$, we combine the tableaux in 
Figures~\ref{fig:forward} and~\ref{fig:reverse} to 
obtain
\[
    T_w = 
{\begin{young}[c]
 1  & 2 & 6  & 10  & 14  &18  \\
 3  & 4 & 8  & 12 & 16 & 20  \\
 5  & 9 & 13 & 17 & 21 & 22  \\
 7  &11 & 15 & 19 & 23 & 24
\end{young}}\ .
\]

Since the definition of $T_w$ is piecewise, it is not immediately
clear that this is always a sensible construction.  We will show
that the constructions in Algorithms~\ref{alg:forward} 
and~\ref{alg:reverse} agree on the diagonal, i.e. 
$T_w^\la+[\diag_i] = T_w^{\Rect/\la-}[\diag_i]$,
for $i=1, \dots, n$.  This is the first step in proving:

\begin{theorem}
\label{thm:combine}
$T_w$ is a standard Young tableau.  Moreover, the definition of $T_w$ is 
independent of the choice of diagonal $\la+/\la-$.
\end{theorem}

Our main result states that this construction gives the minimal
orbits of promotion.

\begin{theorem}
\label{thm:bijection}
The map $w \mapsto T_w$ defines a bijection between $S_n$ and $\calO_n$.
Specifically the following hold:
\begin{packedenumi}
\item  For all $w \in S_n$, $\promote T_w= T_{wc}$, where 
$c = (1\ n\ n{-}1\ \dots\ 2)$.
In particular $T_w \in \calO_n$.
\item If $w, w' \in S_n$ and $w(i) \neq w'(i)$, then 
$T_w[\diag_i] \neq T_{w'}[\diag_i]$.  In particular $w \mapsto T_w$
is injective.
\item For each $T \in \calO_n$, consider the function
$w : \{1, \dots, n\} \to \{1, \dots, n\}$
such that $ w(i) \equiv T[\diag_i] \pmod n$, for $i =1, \dots, n$.  
We have $w \in S_n$, and $T_w = T$.  In particular $w \mapsto T_w$
is surjective.
\end{packedenumi}
\end{theorem}

The rest of this paper is organized as follows.  In 
Section~\ref{sec:strategy} we develop a reduction strategy for 
proving Theorems~\ref{thm:welldefined} and~\ref{thm:combine}.
This strategy is implemented in Section~\ref{sec:descents}, where we
prove two lemmas: the first reducing the problem to one we can solve, 
and the second solving it.  All three 
theorems 
are proved in Section~\ref{sec:proofs}.  Finally, in 
Section~\ref{sec:remarks} 
we discuss some additional facts that are true, and some that
we would like to be true.


\section{Strategy}
\label{sec:strategy}

To prove Theorem~\ref{thm:welldefined}, we need to formulate it
in a slightly different way.  As with ordinary rectification, each 
possible sequence of choices of boxes in Algorithm~\ref{alg:forward} 
can be encoded by a standard Young tableau $U \in \SYT(\la-)$,
by putting $U[\one] := |\la-|+1-k$ if $\one$ is the box chosen
in the $k$\nth iteration of the loop.  Let $T_w^U$ denote
the result of applying Algorithm~\ref{alg:forward} with
order of boxes encoded by $U$.  For example, Figure~\ref{fig:forward}
computes $T_{3142}^U$, with
\[
    U \ =\ %
\begin{young}[c]
1 & 3 & 6 & 7 \\
2 & 4 & 9 \\
5 & 8
\end{young}\,.
\]
Theorem~\ref{thm:welldefined}
states that $T_w^U$ is independent of $U \in \SYT(\la-)$.

\begin{figure}[tb]
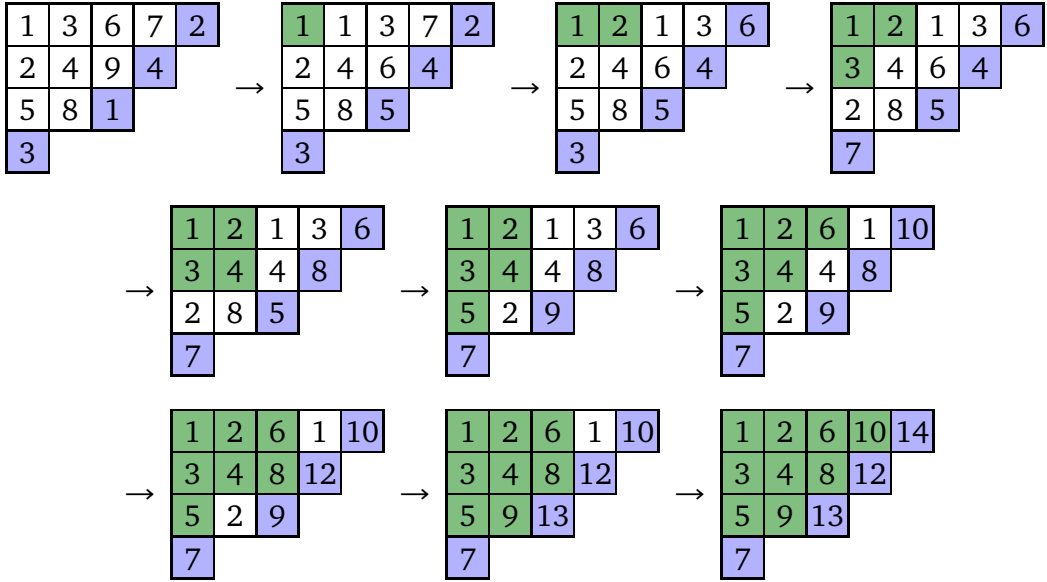

\begin{gather*}
{\begin{young}[c]
 1 & 3 & 6 & 7 &!2 \\
 2 & 4 & 9 &!4  \\
 5 & 8 &!1 \\
!3
\end{young}}
\ \to \ %
{\begin{young}[c]
?1  & 1 & 3 & 7 &!2 \\
 2 & 4 & 6 &!4  \\
 5 & 8 &!5 \\
!3
\end{young}}
\ \to \ %
{\begin{young}[c]
?1  & ?2 & 1 & 3 &!6 \\
 2 & 4 & 6 &!4  \\
 5 & 8 &!5 \\
!3
\end{young}}
\ \to \ %
{\begin{young}[c]
?1  & ?2 & 1 & 3 &!6 \\
?3 & 4 & 6 &!4  \\
 2 & 8 &!5 \\
!7
\end{young}}
\\[1.5ex]
\ \to \ %
{\begin{young}[c]
?1  & ?2 & 1 & 3 &!6 \\
?3 & ?4  & 4 & !8  \\
 2 & 8 &!5 \\
!7
\end{young}}
\ \to \ %
{\begin{young}[c]
?1  & ?2 & 1 & 3 &!6 \\
?3 & ?4  & 4 & !8  \\
?5 & 2 & !9 \\
!7 
\end{young}}
\ \to \ %
{\begin{young}[c]
?1  & ?2 & ?6 & 1 &!10 \\
?3 & ?4  & 4 & !8  \\
?5 & 2 & !9 \\
!7 
\end{young}}
\\[1.5ex]
\ \to \ %
{\begin{young}[c]
?1  & ?2 & ?6 & 1 &!10 \\
?3 & ?4  & ?8 & !12  \\
?5 & 2 & !9 \\
!7 
\end{young}}
\ \to \ %
{\begin{young}[c]
?1  & ?2 & ?6 & 1 &!10 \\
?3 & ?4  & ?8 & !12  \\
?5 & ?9 & !13 \\
!7 
\end{young}}
\ \to \ %
{\begin{young}[c]
?1  & ?2 & ?6 & ?10 &!14 \\
?3 & ?4  & ?8 & !12  \\
?5 & ?9 & !13 \\
!7 
\end{young}}
\end{gather*}
\caption{An alternative way to compute $T_{3142}^\la+$, 
using the 
same order for the boxes of $\la-$ as the example in Figure~\ref{fig:forward}.}
\label{fig:sequence}
\end{figure}

The steps of a rectification-type algorithm can be
performed in a variety of different but equivalent orders.  
In particular, instead
of sliding the entries of $U$ through $T$, from largest to smallest,
one can reverse-slide the entries of $T$ though $U$, from smallest to
largest, (see \cite{BSS} for full details).  
Figure~\ref{fig:sequence} illustrates this in the context 
of Algorithm~\ref{alg:forward}, using the example from
Figure~\ref{fig:forward}.

This perspective not only gives a reformulation of
Algorithm~\ref{alg:forward}, but allows us generalize it 
to inputs that are not permutations.
Let $\sigma = \sigma_1\sigma_2\sigma_3 \dots$ be an infinite 
sequence, with $\sigma_k \in \{1, \dots, n\}$ for $k =1, 2, 3, \dots$.
We construct a sequence $\boxseq_\sigma^U = \one_1\, \one_2\, \one_3\, \dots$
of boxes of $\la+$, as follows.

\begin{algorithm} 
\label{alg:sequence}
\emph{INPUT:} The pair $(\sigma, U)$.
\begin{algorithmic}
\State Let $U_0 := U$;
\For {$k = 1, 2, 3, \dots$}
\State Let $U_k$ be the tableau obtained by reverse-sliding 
    box $\diag_{\sigma_k}$ through $U_{k-1}$;
\State Define $\one_k$ to be the final position of the sliding path;
\State Delete the entry in $\diag_{\sigma_k}$ from $U_k$, if one exists;
\EndFor
\State \Return $\boxseq_\sigma^U := \one_1\, \one_2\, \one_3\, \dots.$
\end{algorithmic}
\end{algorithm}

We use this sequence to
define a function $\delta_\sigma^U : \{1, \dots, n\} \to \ZZ_{\geq 0}$,
\[
  \delta_\sigma^U(i) := 
  \#\{k  \mid \sigma_k = i \text{ and }\one_k \neq \diag_i\}\,,
\]
where $\boxseq_\sigma^U = \one_1\,\one_2\,\one_3\,\dots$.  This function
will be key in proving Theorem~\ref{thm:combine}.  If
$\boxseq_\sigma^U$ is independent of $U \in \SYT(\la-)$, we write
$\boxseq_\sigma^\la+ := \boxseq_\sigma^U$, and $\delta_\sigma^\la+ := 
\delta_\sigma^U$.

Strictly speaking, Algorithm~\ref{alg:sequence} is not a proper algorithm, 
in that it does not terminate;
however, we are really only interested in a finite part
of $\boxseq_\sigma^U$.  Once $k$ is sufficiently large, we have 
$\one_k = \diag_{\sigma_k}$.  For $N \geq 0$, denote the truncation
of a sequence at its $N$\nth
term by
$\trunc_N(a_1 a_2 a_3 \dots) := (a_1 a_2 \dots a_N)$.

For $w \in S_n$,  define $w^*$ to be the repeating sequence
\[
w^* := a_1a_2 \dots a_n\, a_1 a_2 \dots a_n\, a_1 a_2 \dots,
\]
where $a_1a_2 \dots a_n$ is the word representing $w^{-1}$ in
one line notation (i.e. $a_i = w^{-1}(i)$ for $i=1, \dots, n$).  
The following proposition precisely states 
the relationship between
Algorithms~\ref{alg:forward} and~\ref{alg:sequence}.

\begin{proposition}
\label{prop:reformulate}
Write $\boxseq_{w^*}^U = \one_1\, \one_2\, \one_3\, \dots.$
For each box $\one \in \la+$, $T_w^U[\one]$ is the smallest $k$ 
such that $\one_k = \one$\,.  Hence, for some $N \geq 0$,
$\trunc_N\big(\boxseq_{w^*}^U\big)$ determines $T_w^U$.
In addition, we have 
$T_w^U[\diag_i] = w(i) + n \cdot \delta_{w^*}^U(i)$.
\end{proposition}

\begin{proof}
The first two statements are simply a more precise formulation of the
remarks at the beginning of this section.  For the third, note
that the $k$\nth term of $w^*$ is equal to $i$, for 
$k = w(i), w(i)+n, w(i)+2n, \dots.$
By the definition of $\delta_\sigma^U$, the 
$(\delta_{w^*}^U+1)$\nth term in this sequence is the smallest
$k$ such that $\one_k = \diag_i$.  The former is 
$w(i) + n\cdot \delta_{w^*}^U(i)$, and the latter is $T_w^U[\diag_i]$.
\end{proof}

Partially order the boxes of $\la+$: let $\one \leq \one'$ if $\one'$
is both weakly right of $\one$ and weakly above $\one$.  

\begin{proposition}
\label{prop:descents}
Write $\boxseq_{\sigma}^U = \one_1\, \one_2\, \one_3\, \dots.$ 
If $\sigma_k < \sigma_{k+1}$ then $\one_k < \one_{k+1}$;
if $\sigma_k > \sigma_{k+1}$ then $\one_k > \one_{k+1}$.
\end{proposition}

\begin{proof}
This follows from the fact that the \emph{jeu de taquin} 
preserves horizontal
(and vertical) strips.
\end{proof}

For a sequence
$a_1 a_2 a_3 \dots $ with terms from a partially ordered set
(e.g. the numbers $\{1, \dots, n\}$ or the boxes of $\lambda_+$), 
define 
the \defn{strict Knuth transformations} to be the operations
\[
\kappa_k(a_1 a_2 a_3 \dots ) 
= a_1 a_2 \dots a_{k-1} x y z a_{k+3} a_{k+4} \dots,
\]
where
\[
    xyz = 
\begin{cases} 
a_k a_{k+2} a_{k+1}
&\quad 
\text{if\ \ $a_{k+1} < a_k < a_{k+2}$ 
\ \ or\ \ $a_{k+2} < a_k < a_{k+1}$}
\\
a_{k+1} a_k a_{k+2} 
&\quad 
\text{if\ \ $a_k < a_{k+2} < a_{k+1}$ 
\ \ or\ \  $a_{k+1} < a_{k+2} < a_k$}
\\
\text{undefined}
&\quad 
\text{otherwise.}
\end{cases}
\]
In the third case $\kappa_k(a_1a_2a_3\dots)$ is also undefined.
These are similar to elementary Knuth transformations on sequences, 
except that the inequalities are required to be strict.
We define two sequences $a_1 a_2 a_3 \dots$ and $b_1 b_2 b_3 \dots$
to be \defn{equivalent} if for every $N \geq 0$, there exists a 
finite sequence
$\kappa_{k_1}, \kappa_{k_2}, \dots, \kappa_{k_m}$ of strict 
Knuth transformations such that
\[
\trunc_N \big(\kappa_{k_1} \circ \kappa_{k_2} \circ \dots \circ \kappa_{k_m}
(a_1 a_2 a_3 \dots)\big) = \trunc_N \big(b_1 b_2 b_3 \dots\big)
\,.
\]
When $a_1a_2a_3\dots$ is a sequence of boxes, this 
generalizes of the notion of dual equivalence on tableaux \cite{Hai}.

\begin{proposition}
\label{prop:equivariant}
Let  $\sigma$ be a sequence with terms from $\{1, \dots, n\}$. 
If $\kappa_k(\sigma)$ is defined, then
$\kappa_k(\boxseq_\sigma^U)$ is defined, and
\begin{packedenumi}
\item $\boxseq_{\kappa_k(\sigma)}^U = \kappa_k(\boxseq_\sigma^U)$;
\item $\delta_{\kappa_k(\sigma)}^U(i) = \delta_\sigma^U(i)$, 
for $i =1 ,\dots, n$.
\end{packedenumi}
\end{proposition}

\begin{proof}
Write $\Hsigma := \kappa_k(\sigma)$.
Let $\boxseq_\sigma^U = \one_1\, \one_2\, \one_3\, \dots$,
and let $U_0, U_1, U_2, \dots$ be the sequence of tableaux produced
in Algorithm~\ref{alg:sequence}.
Let $\boxseq_{\Hsigma}^U 
 = \Hone_1\, \Hone_2\, \Hone_3\, \dots$,
and $\HU_0, \HU_1, \HU_2, \dots$ be the corresponding objects
for $\Hsigma$.
Since $\sigma_j = \Hsigma_j$ for $j < k$,
we have $\one_j = \Hone_j$ and $U_j = \HU_j$, for $j < k$.
In particular $U_{k-1} = \HU_{k-1}$.
Given $3$ boxes $\one, \one', \one''$ let $T(\one, \one', \one'')$
denote the standard Young tableau with entries $1$, $2$, and $3$,
in boxes $\one$, $\one'$ and $\one''$ respectively.
The pair of tableaux
$T(\diag_{\sigma_k}, \diag_{\sigma_{k+1}}, \diag_{\sigma_{k+2}})$
and
$T(\diag_{\Hsigma_k}, \diag_{\Hsigma_{k+1}}, \diag_{\Hsigma_{k+2}})$
form a dual equivalence class. 
It follows from \cite[Corollary 2.8]{Hai} that $U_{k+2} = \HU_{k+2}$,
and since $\sigma_j = \Hsigma_j$ for $j > k+2$ we have
$\one_j = \Hone_j$ for $j > k+2$.
By \cite[Lemma 2.3]{Hai}, $T(\one_k, \one_{k+1}, \one_{k+2})$
and $T(\Hone_k, \Hone_{k+1}, \Hone_{k+2})$
also form a dual equivalence class, which implies that the three-term
sequences
$(\one_k, \one_{k+1}, \one_{k+2})$ 
and $(\Hone_k, \Hone_{k+1}, \Hone_{k+2})$
are related by a strict Knuth transformation.  This proves (i),
and (ii) is straightforward.
\end{proof} 

This leads to our strategy for proving 
Theorems~\ref{thm:welldefined} and~\ref{thm:combine}, which
is outlined in the next proposition.
For the second statement in 
Theorem~\ref{thm:combine} we will need to consider 
how the constructions in Algorithms~\ref{alg:forward} 
and~\ref{alg:reverse} are related for different choices of diagonal.
This is facilitated by the following definition.
Let $\primela+/\primela-$ be another diagonal of $\Rect$\,.
Let $\one_1\,\one_2\,\one_3\,\dots$ be a sequence of boxes of $\la+$,
and let $\one'_1\,\one'_2\,\one'_3\,\dots$ be a sequence of boxes of 
$\primela+$.  We say that these sequences are \defn{compatible} if
$\one'_k = \one_k$ whenever $\one'_k \in \la+$ and $\one_k \in \primela+$.

\begin{proposition}
\label{prop:strategy}
Suppose that $\sigma$ is equivalent to $w^*$, and
$\boxseq_\sigma^\la+$ is well-defined
(i.e. $\boxseq_\sigma^U$ is independent $U \in \SYT(\la-)$).  
Then the following are true.
\begin{packedenumi}
\item  
$T_w^\la+$ is well-defined (i.e.
$T_w^U$ is independent of $U \in \SYT(\la-)$).
\item 
$T_w^\la+[\diag_i] = w(i) + n \cdot \delta_\sigma^\la+(i)$.
\item 
Suppose $\primela+$ is obtained from $\la+$ by adding one box.
If
$\boxseq_\sigma^\primela+$ is well-defined and compatible with 
$\boxseq_\sigma^\la+$, then the tableaux
$T_w^\la+$ and  $T_w^{\la+'}$ coincide on $\la+$.
\end{packedenumi}
\end{proposition}

\begin{proof}
Let $U, \HU \in \SYT(\Rect)$. To prove (i), we 
must show that $T_w^U = T_w^\HU$.
By Proposition~\ref{prop:reformulate},
there exists $N \geq 0$ so that $T_w^U$ and $T_w^\HU$ are determined by
$\trunc_N(\boxseq_{w^*}^U)$ and $\trunc_N(\boxseq_{w^*}^\HU)$; 
therefore it is enough to show that the latter two are equal.
Since $w^*$ is equivalent to $\sigma$ there is a sequence
$\kappa_{k_1}, \kappa_{k_2}, \dots, \kappa_{k_M}$ of strict 
Knuth transformations such that
\[
\trunc_N \big(\kappa_{k_1} \circ \kappa_{k_2} \circ \dots \circ \kappa_{k_M}
(\sigma)\big) = \trunc_N \big(w^*\big)
\]
Since $\boxseq_\sigma^\la+ =\boxseq_\sigma^U =\boxseq_\sigma^\HU$,
by Proposition~\ref{prop:equivariant}(i) we have
\[
\trunc_N\big(\boxseq_{w^*}^U\big)
=
\trunc_N \big(\kappa_{k_1} \circ \kappa_{k_2} \circ \dots \circ \kappa_{k_M}
(\boxseq_\sigma^\la+)\big) 
= 
\trunc_N\big(\boxseq_{w^*}^\HU\big)\,,
\]
as required.
Similarly, (ii) follows from
Proposition~\ref{prop:reformulate} and
Proposition~\ref{prop:equivariant}(ii).
For (iii), it is easy to see that $\boxseq_\sigma^{\la+}$
is compatible with $\boxseq_\sigma^{\primela+}$ if and only if 
that $\kappa_k(\boxseq_{\sigma}^{\la+})$ is compatible with
$\kappa_k(\boxseq_{\sigma}^{\primela+})$.  By 
Proposition~\ref{prop:equivariant}(i), 
$\boxseq_{w^*}^\la+$ is compatible with
$\boxseq_{w^*}^\primela+$; the result follows by
Proposition~\ref{prop:reformulate}.
\end{proof}

In the next section we will construct a suitable $\sigma$ for each 
permutation $w$, enabling us to prove Theorems~\ref{thm:welldefined}
and~\ref{thm:combine}.  The construction of $\sigma$ is based on
the \emph{cyclage} operation of Lascoux and Sch\"utzenberger \cite{LS}.
The following two facts will be used to establish equivalence:

\begin{proposition}
\label{prop:periodicequivalence}
Let $w, \Hw \in S_n$ be two permutations.  
If $w^{-1}$ and $\Hw^{-1}$ have the same insertion tableau,
then $w^*$ is equivalent to $\Hw^*$.
\end{proposition}

\begin{proof}
Fix $N \geq 0$.
Let $a_1a_2 \dots a_n$ and $b_1b_2 \dots b_n$ be the words representing
$w^{-1}$ and $\Hw^{-1}$ respectively.  Since these words have the same 
insertion tableau, they are related by a finite sequence of elementary 
Knuth transformations, and since $a_i \neq a_j$ for $i \neq j$, each of
these is a strict Knuth transformation.  It follows that
$w^*$ can be transformed into any sequence of the form
\[
   b_1b_2 \dots b_n\,b_1b_2 \dots b_n\ \cdots\ b_1b_2 \dots b_n\,
   a_1a_2 \dots a_n\,a_1a_2 \dots a_n\,a_1a_2\ \dots
\]
using a finite sequence of strict Knuth transformations.
If there are at least $N/n$ copies of $b_1b_2\dots b_n$, then 
truncating
at the $N$\nth term gives $\trunc_N(\Hw^*)$, as required.
\end{proof}

\begin{proposition}
\label{prop:rowstrict}
Let $a_1a_2a_3 \dots$ and $b_1b_2b_3 \dots$ be 
positive integer sequences.
Let $A_k$ be the insertion tableau of the finite word
$a_1a_2 \dots a_k$, and let $B_k$ be the insertion tableau of
$b_1b_2 \dots b_k$.
Suppose there exists a number $M$ such that the following hold: 
$A_M = B_M$;
$a_k = b_k$ for all $k > M$; and $A_k$ and $B_k$ are row-strict 
for all
$k \leq M$.  Then $a_1a_2a_3 \dots$ is equivalent to $b_1b_2b_3 \dots.$
\end{proposition}

\begin{proof}
Let $A_{k,1}A_{k,2}\dots A_{k,k}$ denote the reading word of $A_k$.
For $k \leq M$ there is a finite sequence of
elementary Knuth transformations taking
\[
   A_{k-1,1}A_{k-1,2} \dots A_{k-1,k-1}a_k \ \mapsto\ %
   A_{k,1}A_{k,2} \dots A_{k,k-1}A_{k,k}\,;
\]
the precise sequence can be found in many references
(e.g. \cite[Section 2.1]{Ful} or \cite[Section 6.1]{LLT}).
It is easy to verify that
if $A_k$ is row-strict, then all of the transformations in this 
sequence are \emph{strict} Knuth transformations.
This shows that $a_1a_2a_3 \dots$ is equivalent to 
\[
  A_{M,1}A_{M,2}\dots A_{M,M}a_{m+1}a_{m+2} \dots
  =
  B_{M,1}B_{M,2}\dots B_{M,M}b_{m+1}b_{m+2} \dots
\]
which, by the same argument, is equivalent to $b_1b_2b_3 \dots$.
\end{proof}


\section{Descent sequences}
\label{sec:descents}

Recall that $i \in \{1, \dots, n-1\}$ is a \defn{descent} of
$w$ if $w(i) > w(i+1)$;  if $w(i) < w(i+1)$, then $i$ is 
an \defn{ascent} of $w$.
Let $\identity \in S_n$ denote the identity element, and let
$w_0 \in S_n$ denote the long word.
The \defn{major index} of $w$ is defined
to be the sum of the descents of $w$.  For example, 
$\identity$ is the unique permutation with major index $0$, 
$w_0$ is the unique permutation with major index $n(n-1)/2$.

\begin{lemma}
\label{lem:equivalent}
Let $d_1 > d_2 > \dots > d_t$ be the descents of $w$ in decreasing
order, and let $d_i = 0$ for $i > t$.  Then
$w^*$ is equivalent to the sequence
\begin{equation}
\label{eqn:descentseq}
\sigma_{d_1d_2d_3\dots} :=
(d_1{+}1, d_1{+}2, \dots , n,\,
d_2 {+}1 , d_2 {+}2, \dots ,n,\,
d_3 {+}1 , d_3 {+}2, \dots ,n,\,
\dots)\,.
\end{equation}
\end{lemma}

\begin{proof}
For a permutation $w \in S_n$ define $\epsilon_w \in \{1, \dots, n\}$ and
permutations
$\Hw, w' \in S_n$ as follows.
Let $b_1b_2 \dots b_n$ be the reading
word of the insertion tableau of $w^{-1}$.
Let $\epsilon_w := b_1$.  Let $\Hw$ and $w'$ be the permutations 
whose inverses are represented by the words 
$b_1 b_2 \dots b_n$ and $b_2 b_3 \dots b_n b_1$, 
respectively, in one line notation.
Thus 
\[
   \big(\epsilon_w, (w')^*\big)  =
    b_1 b_2 \dots b_n \, b_1 b_2 \dots b_n \, b_1 b_2 \dotsc
     = \Hw^*\,.
\]
By Proposition~\ref{prop:periodicequivalence},
$w^*$ is equivalent to $\Hw^* = \big(\epsilon_w,(w')^*\big)$.
Using this argument repeatedly, $w^*$ is equivalent
to the sequence
\[
   \big(\epsilon_w, \epsilon_{w'}, \epsilon_{w''}, \dots,
      \epsilon_{w^{(M-1)}},
    (w^{(M)})^*\big)
\]
for any $M \geq 0$.
Since Knuth transformations preserve the descents of the inverse 
of a permutation, the descents of $w$ are the same as descents of $\Hw$.

Let $M(w)$ be the
major index of $w_0 w w_0$, which is equal to 
$(n-d_1) + \dots + (n-d_t)$.
Suppose $w \neq \identity$.  Then $\epsilon_w > 1$.  
Since $\Hw(\epsilon_w) = 1$, and $w'(\epsilon_w) = n$,
$\epsilon_w{-}1$ is a descent of $w$ and an ascent of $w'$.
If $\epsilon_w < n$, then $\epsilon_w$ is an ascent of $w$ and a 
descent of $w'$.
For $i \notin \{\epsilon_w{-}1, \epsilon_w\}$,
$i$ is a descent of $w$ if and only if $i$ is a descent of $w'$.
It follows from these remarks that $M(w') = M(w) -1$.

Since $M(w^{(M(w))}) = 0$, $w^{(M(w))} =\identity$.
Thus we have shown $w^*$ is equivalent $(w^\#, \identity^*)$, where
\[
w^\# := 
   \big(\epsilon_w, \epsilon_{w'}, \epsilon_{w''}, \dots ,
      \epsilon_{w^{(M(w)-1)}}\big)\,.
\]
We now show, by induction on $M(w)$, that $(w^\#,\identity^*)$ 
is equivalent to $\sigma_{d_1d_2d_3\dots}$.
If $w = \identity$, the result is
trivial.  Suppose $M(w)>1$ and assume the result is true 
for $w'$.  Then $(w^\#, \identity^*) = (\epsilon_w , (w')^\#, \identity^*)$
is equivalent to $(\epsilon_w, \sigma_{d'_1d'_2d'_3\dots})$,
where
$d'_1 > \dots > d'_{t'}$ are the descents of $w'$, 
and $d'_i = 0$ for $i > t'$.
The arguments above show that $\epsilon_w = d_s{+}1$ for some $s \leq t$;
if $\epsilon_w < n$, then $d'_i = d_i$ for $i \neq s$, 
and $d'_s = \epsilon_w$;
if $\epsilon_w = n$, then $d'_i = d_{i+1}$ for all $i$.
In either case, $\sigma_{d'_1d'_2d'_3\dots}$ is 
obtained from $\sigma_{d_1d_2d_3 \dots}$ by deleting the 
first occurrence of $\epsilon_w$.
With this in mind, it follows readily from 
Proposition~\ref{prop:rowstrict}
that $(\epsilon_w, \sigma_{d'_1d'_2d'_3\dots})$  
is equivalent to $\sigma_{d_1d_2d_3\dots}$.
\end{proof}

\begin{lemma}
\label{lem:welldefined}
Assume that $n$ is the number of columns of\/ $\Rect$\,.
Let $\sigma = \sigma_{d_1d_2d_3 \dots}$ 
be the sequence in~\eqref{eqn:descentseq}. 
\begin{packedenumi}
\item 
$\boxseq_\sigma^\la+$ is well-defined:  the $k$\nth box of this 
sequence is in column $c_k$, where
\[
c_1 c_2 c_3 \dotsc  := 
(1, 2, \dots, n{-}d_1, \,
1, 2, \dots, n{-}d_2, \,
1, 2, \dots, n{-}d_3, \,
\dots)\,.
\]
\item
If $C_i$ is the length of the $i$\nth column of 
$\la+$, then
\[
\delta_\sigma^\la+(i) = C_i
- \#\big\{j \bigmid d_j \geq i\big\} 
+ \#\big\{j \bigmid d_j \geq n{+}1{-}i\big\} 
- 1
\qquad\text{for $i = 1, \dots, n$.}
\]
\item
For any other diagonal $\primela+/\primela-$ of\/ $\Rect$\,, 
$\boxseq_\sigma^\primela+$ is compatible with $\boxseq_\sigma^\la+$.
\end{packedenumi}
\end{lemma}

\begin{proof}
We introduce the notation 
\[[i,j] := i + \sum_{s=1}^{j-1} (n-d_s)\,,
\qquad\text{for $1 \leq i \leq n{-}d_j$ and $j \geq 1$}. 
\]
Thus $c_{[i,j]} = i$ and $\sigma_{[i,j]} = d_j+ i$ for all $i, j$.
Fix $U \in \SYT(\la-)$, 
and write $\boxseq_\sigma^U = \one_1\ \one_2\ \one_3 \dots.$
Suppose $\one_k$ is in column $e_k$.
To prove (i), we need to show that $e_{[i,j]} = i$
for all $i,j$.  
We will do this by strong induction on $j$.  

Fix $j \geq 0$. 
Assume that $e_{[i,s]} = i$ for 
$1 \leq s \leq j$, $1 \leq i \leq n-d_s$.
Let $d_0 := n$, $j' := j+1$, $p := n - d_j$, $p' := n-d_{j'}$.
We will prove the following: 
\begin{packedenum}
\item[(a)]
$e_{[i,j']} \geq i$ for $1 \leq i \leq p'$; 
\item[(b)]
$e_{[i,j']} = e_{[i-1,j']}+1$ for $p < i \leq p'$
(where $e_{[0,1]} := 0$);
\item[(c)]
$e_{[i,j']} \leq e_{[i,j]}$ for $1 \leq i \leq p$.
\end{packedenum}
These imply that $e_{[i,j']} = i$ for $i = 1,\dots, p'$.

Since 
\[\sigma_{[1,j']}\sigma_{[2,j']}\dots\sigma_{[p',j']}
 = (d_{j'}{+}1 < d_{j'}{+}2 < \ldots < n)\,,
\]
by Proposition~\ref{prop:descents}
we have 
$\one_{[1,j']} < \one_{[2,j']} < \dots < \one_{[p',j']}$.
Because of the order in which the slides are performed,
$\one_{[i+1,j']}$ cannot be above $\one_{[i,j]}$ and in the same column.
It follows that 
$e_{[1,j']} < e_{[2,j']} < \dots < e_{[p',j']}$,
which proves (a).
Since all boxes $\one_1\,, \dots,\,\one_{[p,j]}$ are
in the first $p$ columns, any box $\one_{[i,j']}$ which is not in
the first $p$ columns must be in the first row.  In particular
this applies when $i > p$, which proves (b).

If $p' = n$, then (c) follows immediately from (a).  To complete
the proof of (i), suppose that $p' < n$.  Then $d_j > d_{j'}$.
Let $a_i := \sigma_{[i,j]} = d_j+i$ 
and let $b_i := \sigma_{[i,j']} = d_{j'}+i$.
Consider sequence obtained from $\sigma$ by changing the subsequence
of length $2p$ starting at $\sigma_{[1,j]}$ from
\[
(a_1< a_2 < \ldots < a_p > b_1< b_2 < \dots < b_p)
\]
to
\[
   (a_1 > b_1 < a_2 > b_2 < \ldots < a_p > b_p)\,.
\]
This transformation can be realized as
$K_{p-1} \circ K_{p-2} \circ \dots \circ K_1(\sigma)$,
where
\[
K_i := \kappa_{[1,j]+2i-2} \circ \kappa_{[1,j]+2i-1} \circ \kappa_{[1,j]+2i} \dots 
\circ \kappa_{[1,j]+i+p-3}
\]
is the composition of strict Knuth transformations that
moves $b_i$ next to $a_i$.  
For example, $K_1$ performs the following sequence of transformations:
\begin{align*}
(a_1 <  &\ldots< a_{p-2} < a_{p-1}< a_p > \boldsymbol{b_1}
< b_2 < \ldots < b_p)
\\
& 
\mapsto \ %
(a_1 <  \ldots < a_{p-2} < a_{p-1} > \boldsymbol{b_1} < a_p
> b_2 < \ldots < b_p)
\\
&
\mapsto \ %
(a_1 <  \ldots < a_{p-2} > \boldsymbol{b_1} < a_{p-1} < a_p
> b_2 < \ldots < b_p)
\\
&\dots\\
& 
\mapsto \ %
(a_1 > \boldsymbol{b_1} < a_2 < \ldots < a_p > b_2 < \ldots < b_p)\,.
\end{align*}
Here we have recorded only the subsequence of length $2p$ starting
at $[1,j]$ --- the remaining terms are unaffected by these transformations.

Let $\alpha_i := \one_{[i,j]}$ and let $\beta_j := \one_{[i,j']}$.
The corresponding subsequence of $\boxseq_\sigma^U$ is
\[
   (\alpha_1 < \alpha_2 < \ldots < \alpha_p > \beta_1 < \beta_2 < \ldots < \beta_p)\,.
\]
By Proposition~\ref{prop:equivariant}, 
$K_{p-1} \circ K_{p-2} \circ \dots \circ K_1(\boxseq_\sigma^U)$ is defined,
and by Proposition~\ref{prop:descents}, each strict Knuth 
transformation must produce a
sequence with the correct descent pattern.
Using these two facts, one can deduce (by a straightforward
inductive argument) 
that for all $r  = 1, \dots, p-1$,
the corresponding subsequence of
$K_r \circ \dots \circ K_1(\boxseq^U_\sigma)$ must be of the form
\[
(\alpha_{q_1} > \gamma_1 < \alpha_{q_2} < \gamma_2 > 
\ldots < \alpha_{q_r} > \gamma_r < \gamma_{r+1} < \ldots < \gamma_p 
> \beta_{r+1} < \ldots < \beta_p)\,,
\]
where $1 \leq q_1 < q_2 < \dots < q_r \leq p$ 
and $(\gamma_1 < \gamma_2  < \cdots < \gamma_p)$ is obtained from
$\alpha_1\alpha_2 \dots \alpha_p$ by replacing $\alpha_{q_i}$ 
replaced by $\beta_i$ for
$i =1, \dots,r$.   In particular, when $r = p-1$, we have $\gamma_p > \beta_p$;
thus $\gamma_p = \alpha_p$, and $q_i = i$ for all $i$.
This shows that
\[
K_{p-1} \circ \dots \circ K_1(\boxseq_\sigma^U)
= (\dots\ \alpha_1 > \beta_1 < \alpha_2 > \beta_2 < \ldots < \alpha_p > \beta_p \ \dots)\,.
\]
The descent pattern of this sequence establishes that $\alpha_i > \beta_i$, 
which proves (c).

For (ii), suppose that $C_i$\nth occurrence of $i$ in
the sequence $c_1c_2c_3 \dots$ occurs at $[s,i]$.
Since the subsequence 
\[
c_{[1,j]} c_{[2,j]}\dots c_{[n-d_j,j]} = (1,2, \dots, n{-}d_j)
\]
excludes $i$ if and only if $d_j \geq n+1-1$, 
$s = C_i + \#\{j \mid d_j \geq n{+}1{-}i\}$.
By (i), $\one_k$ is in column $c_k$, if and only if $k=[j,i]$
for some $j$, and $\one_k = \diag_i$ when $j \geq s$;
therefore
$\delta_\sigma^{\la+}(i)$ is the number of occurrences of $i$ 
in the sequence $\trunc_{[n-d_{s-1},s-1]}(\sigma)$.
Since the subsequence
\[
\sigma_{[1,j]} \dots \sigma_{[n-d_j,j]}
 = (d_j{+}1, d_j{+}2, \dots, n)
\]
excludes $i$ if and only if $d_j \geq i$, 
$\delta_\sigma^\la+(i) = s-1 - \#\{j \mid d_j \geq i\}$, 
as required.

Finally, (iii) follows immediately from (i).
\end{proof}


\section{Proofs}
\label{sec:proofs}
We now prove Theorems~\ref{thm:welldefined}, \ref{thm:combine} 
and \ref{thm:bijection}.

\begin{proof}[Proof of Theorem~\ref{thm:welldefined}]
Since the result is symmetrical with respect to rows and columns,
we may assume, without loss of generality, that $n$ is the number 
of columns of $\Rect$\,.
By Lemma~\ref{lem:equivalent} and Lemma~\ref{lem:welldefined}(i),
$w^*$ is equivalent to a sequence $\sigma$ such that $\boxseq_\sigma^\la+$
is well-defined.  The theorem 
therefore follows from Proposition~\ref{prop:strategy}(i).
\end{proof}

\begin{proof}[Proof of Theorem~\ref{thm:combine}]
Again, assume, without loss of generality, that $n$ is the number 
of columns of $\Rect$\,.  
Using Lemma~\ref{lem:equivalent}, Lemma \ref{lem:welldefined}(ii), 
Proposition~\ref{prop:strategy}(ii), and Proposition~\ref{prop:reformulate},
we compute that
\[
  T_w^\la+[\diag_i] 
   = w(i) + n\cdot\Big(
   C_i -
   \#\{j \mid d_j \geq i\}
   + \#\{j \mid d_j \geq n{+}1{-}i\} - 1 \Big)\,,
\]
where $d_1 > d_2 > \dots > d_t$ are the descents of $w$, and
 $C_i$ is the length of column $i$ in $\la+$.
For a partition $\lambda \subset \Rect$\,,
with row lengths $(\lambda_1, \dots, \lambda_m)$,
let $\lambda^\vee$ denote the partition  with row lengths
$(n-\lambda_m, \dots, n-\lambda_1)$.
If $\one$ is the
box of $\Rect$ in column $i$ and row $j$, let $\one^\vee$ denote the
box in column $n+1-i$ and row $m+1-j$.   For a skew tableau $T$
of shape $\lambda/\mu \subset \Rect$\,, let $T^\vee$ denote the 
tableau of shape $\mu^\vee/\lambda^\vee$ with entries 
$T^\vee[\one] := mn+1-T[\one^\vee]$.
The relationship between Algorithms~\ref{alg:forward}
and~\ref{alg:reverse} is
\begin{equation}
\label{eqn:forwardreverse}
   T_w^{\Rect/\la-} = (T_{w_0ww_0}^\veela-)^\vee\,.
\end{equation}
Note that
$m+1 - C_{i}$ is the length of column $n+1-i$ in $\veela-$,
and 
$n{-}d_1 < n{-}d_2 < \dots < n{-}d_t$ are the descents of $w_0ww_0$.
We compute:
\begin{align*}
T_w^{\Rect/\la-}[\diag_i] 
&= (T_{w_0ww_0}^\veela-)^\vee[\diag_i]  \\
&= mn+1 - T_{w_0ww_0}^\veela-[\diag_i^\vee] 
\\
&= mn+1 - \Big(w_0ww_0(n{+}1{-}i) 
+ n \cdot \delta_{(w_0ww_0)^*}^\veela-(n{+}1{-}i)\Big) \\
&= 
mn+1 - \left(
n{+}1{-}w(i) 
+ n \cdot \left(
\begin{aligned}
(m{+}1{-}C_i)  
+ \#\{j \mid n{-}d_j \geq n{+}1{-}i\}&
\\
-\ \#\{j \mid n{-}d_j \geq i\} &
- 1
\end{aligned}
\right)
\right)
\\
&= w(i) + n\cdot\Big(
   C_i -
   \#\{j \mid n{-}d_j \geq n{+}1{-}i\}
   + \#\{j \mid n{-}d_j \geq i\} - 1 \Big)
\\
&=
T_w^\la+[\diag_i]\,,
\end{align*}
i.e. $T_w^\la+$ and $T_w^{\Rect/\la-}$ agree on $\la+/\la-$.
For each vertically adjacent pair of boxes, either both are in $\la+$
or both are in $\Rect/\la-$.  Thus the agreement on $\la+/\la-$
shows that $T_w$ is column strict.  It also follows now, from
Lemma~\ref{lem:equivalent},
Lemma~\ref{lem:welldefined}(iii) and Proposition~\ref{prop:strategy}(iii),
that $T_w$ is independent of the choice of $\la+$.
Thus we may we assume $\la+ = (n, n{-}1, \dots, 2,1)$, which allows
us to see that $T_w$ is row-strict.

Let $D := \{T_w[\diag_i] \mid i = 1, \dots, n\}$ be the
set of diagonal entries of $T_w$.  Algorithm~\ref{alg:forward}
ensures that $T_w[\diag_i] \equiv w(i) \pmod n$, so $D$
contains one number from each congruence class, modulo $n$.
Since the entries of $T_w^\la+$ are 
\[
\big\{k \geq 1 \bigmid k+nj \in D \text{ for some }j \geq 0\big\}
\]
and the entries of $T_w^{\Rect/\la-}$ are
\[
\big\{k \leq mn \bigmid k-nj \in D \text{ for some }j \geq 0\big\}\,,
\]
we see that every number in $\{1, 2,\dots, mn\}$ is an entry of $T_w$.
Therefore $T_w$ is a standard Young tableau.
\end{proof}

\begin{proof}[Proof of Theorem~\ref{thm:bijection}]
For (i), we may assume that $\la+ = (n, n{-}1, \dots, 2,1)$.
This ensures that the sliding path of promotion on any $T \in \SYTRect$
passes through exactly one box of $\la+/\la-$.
Suppose that $T_w^\la+$ is computed using Algorithm~\ref{alg:forward}
by a sequence of slides whose
first step moves the entry $1$.  After this first step, if we delete 
the entry $1$ and decrement all entries by $1$, we are computing 
$T_{wc}^\la+$ instead.  This can be done immediately, or at any
point during  Algorithm~\ref{alg:forward}.
Compare this with the behaviour of  $\promote : T_w \to \promote T_w$ 
on the entries in $\la+$.  Suppose the sliding path passes through
$\la+/\la-$ at $\diag_s$.  
Since the first two steps of promotion delete 
the entry $1$ and decrement all entries by $1$, this produces
the penultimate step in the construction of $T^{\la+}_{wc}$.  Next, we
slide the empty box in the upper-left corner of $\Rect$ through the
tableau, which is the almost the same as the final of step 
in the construction of $T^{\la+}_{wc}$, except that
do not yet know what number will appear in $\promote T_w[\diag_s]$.
This shows that for all boxes of $\la+$, with
the possible exception of $\diag_s$, $T_{wc}$ coincides with 
$\promote T_w$.  Note that $s$ is the unique number such that
$T_{wc}[\diag_s] - T_{w}[\diag_s] \neq -1$.

Applying the same argument to $\Rect/\la-$ and the sliding path of
$\promote^{-1} : T_{wc} \to \promote^{-1} T_{wc}$, we see that
with the possible exception of one box $\diag_{s'}$,
$T_w$ coincides with 
$\promote^{-1} T_{wc}$ on $\Rect/\la-$.
Since $s'$ is the unique number such that 
$T_{wc}[\diag_{s'}] - T_{w}[\diag_{s'}] \neq -1$, we must have
$s = s'$.  This shows that these two sliding paths
are in fact inverse to each other, and hence $\promote T_w = T_{wc}$.

Since $T_w[\diag_i] \equiv w(i) \pmod n$,
(ii) is immediate.

To prove (iii), we use another reformulation
of Algorithm~\ref{alg:forward}.

\begin{algorithm}
\label{alg:likepromotion} \emph{INPUT:} A permutation $w \in S_n$. 
\begin{algorithmic}
\State
Begin with $T := \varnothing$, the empty tableau, and $\mu := \Rect$\,;
\While{$\mu$ is not the empty partition}
\State  Choose a corner box $\one \in \mu$;
\If {$\one = \diag_i$ for some $i$} 
   \State Set $T[\diag_i] := w(i)$;
\EndIf
\If {$\one \in \la-$}
\State Let $T'$ be the tableau obtained by sliding $\one$ through $T$; 
\State If the final position of the sliding path is $\diag_i$, then set 
$T'[\diag_i] := T[\diag_i] + n$; 
\State  Set $T := T'$;
\EndIf
\State Delete the box $\one$ from $\mu$;
\EndWhile
\State \Return the resulting tableau, $T_w^\la+ := T$.
\end{algorithmic}
\end{algorithm}

It is clear that Algorithm~\ref{alg:likepromotion} is equivalent to
Algorithm~\ref{alg:forward}: when $\one \notin \la+$ nothing happens;
when $\one \in \la+/\la-$ we create the intial entries of $T$;
when $\one \in \la-$ we proceed exactly as before.

Suppose $T \in \calO_n$.  For $i,k = 1, \dots, n$, let
$\Delta_{ik} := \promote^kT [\diag_i] - \promote^{k-1}T[\diag_i]$.
Thus $\Delta_{ik} \geq 0$ if and only if
the sliding path of $\promote : \promote^{k-1} T \mapsto \promote^k T$,
passes through $\diag_i$, and $\Delta_{ik} = -1$ otherwise.
The former can happen for at most one value of $i$.
Since $\promote^n T = T$, 
\[
   \Delta_{i1} + \Delta_{i2} + \dots \Delta_{in} = 0\,.
\]
Therefore, for each $i$, 
there must be at least one $k$ such that $\Delta_{ik} \geq 0$.
It follows that for each $k$ there is exactly one $i$ such that
$\Delta_{ik} \geq 0$, and for each $i$ there is exactly one $k$ such that
$\Delta_{ik} \geq 0$.  From this we see that 
if $\Delta_{ik} \geq 0$ then $\Delta_{ik} = n-1$, and therefore,
for all $k \geq 0$, 
\[
   \promote^kT [\diag_i] - (mn-k) \equiv w(i)  \pmod n\,.
\]

For $k=1, \dots, mn$, construct a tableau $T_k$ by starting with
$\promote^k T$,
subtracting $mn-k$ from all entries, and deleting any entries for which
the result is
less than or equal to $0$.  Let $\mu_k$ be the shape formed by the
unfilled boxes of $T_k$.  Thus $T_0$ is empty, $\mu_0 = \Rect$\,, and 
$T_k$ is obtained by $T_{k-1}$ as follows:
let $\one_k \in \mu_{k-1}$ be the corner of $\mu_{k-1}$ on the sliding
path of 
$\promote : \promote^{k-1}T \mapsto \promote^k T$;
slide $\one_k$ through $T_{k-1}$; add entry $k$ in the lower-right corner
of $\Rect$.

Let $w: \{1, \dots, n\} \to \{1, \dots, n\}$ be the function defined 
in the statement of (iii).
If $T_k[\diag_i]$ is non-empty, then $T_k[\diag_i] \equiv w(i) \pmod n$.
Since $\Delta_{ik} < n$ for all $k$, if 
$\one_k = \diag_i$ then $T_k[\diag_i] \leq n$, i.e. $T_k[\diag_i] = w(i)$;
and if $\one_k \in \la-$ and the sliding path of $\one_k$ passes 
through $\diag_i$, then $T_k[\diag_i] = T_{k-1}[\diag_i] + n$.
Thus if we restrict the sequence $T_0, T_1, \dots, T_{mn}$ to $\la+$,
we obtain precisely a sequence of tableaux produced by 
Algorithm~\ref{alg:likepromotion}.  Since $\promote^{mn}T = T$,
this shows that $T_w^\la+$ is the restriction of $T$ to $\la+$.
Since $T$ has no repeated entries, $w \in S_n$.
By a similar argument $T_w^{\Rect/\la-}$ is the restriction of $T$ to 
$\Rect/\la-$.  Thus $T = T_w$.  
\end{proof}


\section{Remarks}
\label{sec:remarks}

Here is another way to compute $T_w^\la+$.  Assume that
$n$ is the number of rows of $\Rect$\,.
Define the \defn{augmented word} of $w$ to be:
\begin{align*}
\aug(w)\, := \,
  &w(1), w(1)+n, w(1)+2n, \dots w(1)+(m-1)n, \\
  &w(2), w(2)+n, w(2)+2n, \dots w(2)+(m-1)n, \\
  & \; \dots \\
  &w(n), w(n)+n, w(n)+2n, \dots w(n)+(m-1)n\,. \\
\end{align*}

\begin{theorem}
\label{thm:insertion}
$T_w^\la+$ is the restriction of the insertion tableau of $\aug(w)$
to $\la+$.
\end{theorem}

\begin{proof}
Let $\primela+ := (m,m+1,m+2, \dots, m+n-1)$, and
compute $T_w^{\primela+}$ using Algorithm~\ref{alg:forward}.
Choose a sequence of boxes beginning with $m-1$ boxes from
row $n$, followed $m-1$ boxes from row $n-1$, and so on.
(After $m-1$ boxes in from row $1$, the last $n(n-1)/2$ boxes may 
be taken in any order.)
The first $(m-1)n$ slides produce a tableau whose reading word is
$\aug(w)$.  Therefore if we restrict $T_w^{\primela+}$ to entries
$1, 2, \dots, mn$, we obtain the insertion tableau of $\aug(w)$.
By Theorem~\ref{thm:combine}, $T_w^\la+$ can be obtained as
the restriction of $T_w^\primela+$ to $\la+$. Since the entries of
$T_w^\la+$ are a subset of $\{1, 2, \dots, mn\}$, the result follows.
\end{proof}

Theorem~\ref{thm:insertion} provides an alternate definition 
of $T_w^\la+$. It has the advantage of being well-defined, and 
Theorem~\ref{thm:combine} can be proved by using Greene's 
theorem \cite{Gre} to compute the entries $T[\diag_i]$,
(see \cite[Section 5.2]{Rhe}). 
Unfortunately, things start to break down at the proof of
Theorem~\ref{thm:bijection}, which is
intimately connected to Algorithm~\ref{alg:forward}.
The problem is that although the first and last steps of 
Algorithm~\ref{alg:forward} are related to $\aug(w)$,
the intermediate steps may not be.  
For instance, if $T$ is a tableau from one of the intermediate steps 
it is tempting to define $\aug(T)$ to be the tableau obtained by 
adding entries
$T[\diag_i]+n, T[\diag_i]+2n, \dots w(i)+(m-1)n$.
to the right of $\diag_i$.  Unfortunately, it is \emph{not true}
the Knuth class of $\aug(T)$ is invariant for all $T$.  
There are a number of variations on this idea, and none of them
appear to work.
We do not know how to construct an invariant of
Algorithm~\ref{alg:forward}, analogous to the Knuth class of
the reading word.
In particular, the intermediate tableaux in Algorithm~\ref{alg:forward} 
are not produced by ordinary \emph{jeu de taquin} 
in any seemingly obvious way.
This makes it difficult to prove Theorem~\ref{thm:bijection}, if one 
takes Theorem~\ref{thm:insertion} as the definition of $T_w$.

Another way in which our situation behaves quite differently from
ordinary rectification concerns dual equivalence.
Consider a generalization of Algorithm~\ref{alg:sequence}, in which
we allow $U \in \SYT(\la-/\mu)$ to be a skew shape, but otherwise
the algorithm is performed the same way.
This generalization \emph{does not} have the property that
$\boxseq_{w^*}^U = \boxseq_{w^*}^\HU$, when $U$ is dual equivalent
to $\HU$.  If this were true, it would provide a more straightforward 
proof of
Theorem~\ref{thm:welldefined}.  We do not know a set of elementary 
relations that generate the equivalence relation $U \sim \HU$ $\iff$ 
$\boxseq_{w^*}^U = \boxseq_{w^*}^\HU$ for all $w \in S_n$. 

Despite the aforementioned difficulties, 
Theorem~\ref{thm:insertion} can be used as a definition of
$T_w^\la+$ when $\la+$ is an \emph{arbitrary} partition
with at most $n$ rows --- even in cases where Algorithm~\ref{alg:forward}
does not make sense.
In particular, we can sometimes use this idea to define 
$T_w$, when $m < n$.
We illustrate this with an example.  
Take $n=3$, $m=2$, $\la+ = 211$,
$\la- = 1$, and $w = 132$.  The insertion tableau of $\aug(w)$ is
\begin{align*}
{\begin{young}[c]
1 & 2 & 5 \\
3 & 6 \\
4 
\end{young}}
&\qquad\Longrightarrow\qquad
T_w^\la+ = 
{\begin{young}[c]
1 & 2  \\
3 \\
4 
\end{young}}\ \ .
\\
\intertext{Similarly, using~\eqref{eqn:forwardreverse} as the definition, 
we compute:}
T_w^{\Rect/\la-} = 
{\begin{young}[c]
, & 2  \\
3 & 5 \\
4 & 6
\end{young}}
&\qquad\Longrightarrow\qquad
T_w =
{\begin{young}[c]
1 & 2  \\
3 & 5 \\
4 & 6
\end{young}}\ \ ,
\end{align*}
%
which, indeed, has order $3$ under promotion.
This is very suggestive, 
but it is unclear what to do with the proof 
of Theorem~\ref{thm:bijection}, when $m<n$.  


In the thesis \cite{Rhe}, the second author observed that
a procedure based on rectification can be used construct the
set $\calO_{mn/2}$, when one of $m$, $n$ is even.
%
%
In this case, other bijections are known,
(see \cite[Proposition 3.10]{Pur-ribbon}); 
however it is not obvious that they are equivalent.  
This provides a new perspective, and gives further hints that the methods
introduced in this paper may apply beyond the case of minimal 
orbits.


\bigskip

\footnotesize%
   \textsc{Combinatorics and Optimization Department, 
       University of Waterloo, 200 University Ave. W.  Waterloo, 
       ON, N2L 3G1, Canada.} \texttt{kpurbhoo@uwaterloo.ca}. 

   \medskip

   \textsc{Operations Research Center, MIT, 77 Mass Ave., Building E40--130,
     Cambridge, MA, 02139-4307, USA.}
\texttt{donguk@mit.edu}.
\end{document}